\documentclass[a4paper,11pt, reqno]{amsart}
\usepackage[cm]{fullpage}
\usepackage{hyperref}

\usepackage{mymacros}

\usepackage{enumitem}

\title{Toric degenerations of Calabi--Yau complete intersections and metric SYZ conjecture}
\author{Keita Goto}
\address{
Department of Mathematics, National Taiwan University,
 No.~1, Sec.~4, Roosevelt Rd, Da’an District, Taipei City 10617, 
Taiwan
}
\email{k.goto.math@gmail.com}
\author{Yuto Yamamoto}
\address{
RIKEN iTHEMS, Wako, Saitama 351-0198, Japan}
\email{yuto.yamamoto@riken.jp}
\date{}
\begin{document}
\begin{abstract} 
We consider a toric degeneration 
$\scX$
of Calabi--Yau complete intersections of Batyrev--Borisov in the Gross--Siebert program. 
For the toric degeneration $\scX$, we study the real  Monge--Amp\`{e}re equation corresponding to the non-archimedean Monge--Amp\`{e}re equation that yields the non-archimedean Calabi--Yau metric.
Our main theorem describes the real Monge--Amp\`{e}re equation in terms of tropical geometry and proves the metric SYZ conjecture for
the toric degeneration 
$\scX$ supposing the existence of its solution.
\end{abstract}

\maketitle

\section{Introduction}

The SYZ conjecture predicts that a Calabi--Yau manifold admits the structure of a special Lagrangian torus fibration, and the mirror Calabi--Yau manifold is obtained by taking its dual fibration \cite{MR1429831}.
Its weaker version called the \emph{metric SYZ conjecture} claims as follows:

\begin{conjecture}[{cf.~\cite[Conjecture 1.1]{MR4688155}}]\label{cj:syz}
    Let $\scX=\lc \scX_t \rc_t$ be a maximally degenerate family of polarized Calabi--Yau manifolds of dimension $d$ over the punctured disc $\mathbb{D}^\ast:=\{t\in \bC^\times \ |\ |t| < 1 \}$.
    Then for any $0<\varepsilon \ll 1$, there exists $\delta>0$ such that if $|t|<\delta$ then there exists a special Lagrangian $\bT^d$-fibration on an open subset of the fiber $\scX_t$, whose normalized Calabi--Yau volume is greater than or equal to $1-\varepsilon$.
\end{conjecture}

The metric SYZ conjecture was first proved in \cite{MR4460593} for (a subsequence of) the Dwork family.
It was also proved by Li \cite{MR4688155} that the metric SYZ conjecture is reduced to showing a certain condition called 
the \emph{NA MA-real MA comparison property}.
It requires that there exists an snc-model such that the potential of a solution to the non-archimedean Monge--Amp\`{e}re equation corresponding to the complex Monge--Amp\`{e}re equation giving the Calabi--Yau metric is constant along fibers of the associated Berkovich retraction (cf.~\cite[Definition 3.3]{MR4688155}).
The NA MA-real MA comparison property was confirmed for families of Calabi--Yau hypersurfaces in a projective space
in \cite{MR4692438, PS22}, for families of Calabi--Yau hypersurfaces in certain smooth toric Fano varieties in \cite{MR4701493},
and for families of polarized abelian varieties in \cite{MR4692382}.

The degenerate families of Calabi--Yau manifolds in \cite{MR4692438, PS22, MR4701493}, for which the NA MA-real MA comparison property is confirmed are toric degenerations (in the sense of the Gross--Siebert program \cite{MR2213573, MR2669728, MR2846484}) of Calabi--Yau hypersurfaces that degenerate to the toric boundaries of the ambient smooth toric varieties.
In this article, we study the metric SYZ conjecture for more general toric degenerations.
To be specific, we consider toric degenerations of Calabi--Yau complete intersections of Batyrev--Borisov \cite{MR1463173}, whose special fibers may not be the toric boundaries of the ambient toric varieties.

\subsection{Main result}\label{sc:main-result}

Let $d$ and $r$ be positive integers, and $M$ be a free $\bZ$-module of rank $d+r$.
We set $N:=\Hom(M, \bZ), M_\bR:=M \otimes_\bZ \bR$, and $N_\bR:=N \otimes_\bZ \bR$.
We have the natural pairing
\begin{align}
\la -,- \ra \colon M_\mathbb{R} \times N_\mathbb{R}\to \mathbb{R}.
\end{align}
Let $\Delta \subset M_\bR$ be a reflexive polytope, and $\Delta^\ast := \lc n \in N_\bR \relmid \la \Delta, n \ra \geq -1 \rc$ be its polar polytope.
Let further $\Delta =\Delta_1+ \cdots + \Delta_r$ be a nef partition of $\Delta$, and $\nabla=\nabla_1 + \cdots + \nabla_r \subset N_\bR$ be its dual nef partition (cf.~\cite[Section 4]{MR1463173}).
We take coherent subdivisions $\Sigma' \subset N_\bR, \Sigmav' \subset M_\bR$ of the normal fans $\Sigma \subset N_\bR, \Sigmav \subset M_\bR$ of $\Delta, \nabla$ so that the fan polytopes remain to be $\Delta^\ast, \nabla^\ast$.
We further consider integral piecewise linear functions 
\begin{align}\label{eq:hh}
h \colon N_\bR \to \bR, \quad \check{h} \colon M_\bR \to \bR,
\end{align}
which are strictly convex on the fans $\Sigma'$ and $\Sigmav'$ respectively.
With these data, Gross \cite{MR2198802} constructed a toric degeneration $\mathcal{X}$ of $d$-dimensional Calabi--Yau complete intersections over the affine line $\mathbb{A}^1:=\Spec \bC \ld t \rd$, and its dual intersection complex $B_\nabla^{\check{h}}$.
The precise construction will be recalled in \pref{sc:setting}.
Let $K:=\bC \lbb t \rbb)$.
The ambient space of the generic fiber $X:=\mathcal{X}\times_{\mathbb{A}^1}K$ the toric degeneration $\scX$ is the toric variety $X_{\Sigma'}$ over $K$ associated with the fan $\Sigma'$.
Let $L_h$ be the line bundle on $X_{\Sigma'}$ associated to the integral piecewise linear function $h$ of \eqref{eq:hh}.
By abuse of notation, we also write its restriction to $X$ as $L_h$.

The integral piecewise linear functions \eqref{eq:hh} also give rise to the tropical spheres $A_h\subset M_\bR, B_{\check{h}}\subset N_\bR$ of \cite{MR2187503}.
The definition will be recalled in \pref{sc:sphere}.
They are polytopal complexes consisting of $d$-dimensional polytopes and their faces, and have the natural Lebesgue measures $\mu, \nu$ respectively.
We set the normalized Lebesgue measure 
\begin{align}\label{eq:normalized}
\tilde{\nu}:=\frac{\nu}{\nu(B_{\check{h}})}
\end{align}
on $B_{\check{h}}$.
Let $(X^{\an}, L_h^{\an})$ be the analytification of $(X,L_h)$ over the discrete valued field $K$. 
The essential skeleton $\Sk (X)\subset X^{\an}$ of $X$ (cf.~\cite[Definition 4.10]{MR3370127}) is identified with the dual intersection complex $B_\nabla^{\check{h}}$ of the toric degeneration $\scX$ by \cite[Theorem 1.1(1)]{Yam24}, and it is further equal to the tropical sphere $B_{\check{h}}$ by \cite[Proposition 5.4]{Yam21}.

The non-archimedean Monge--Amp\`{e}re equation we would like to consider is written as
\begin{align}\label{eq:NA-MA}
  {\rm MA} (\Phi) = (L_h^d)\tilde{\nu}  
\end{align}
for semi-positive metrics $\Phi$ on $L_h^{\an}$, where the left hand side denotes the non-archimedean Monge--Amp\`{e}re measure of the metric $\Phi$ (cf.~\cite[{D\'{e}finition 5.6.7}]{CLD12}).
We define $\scP\lb \Delta_h \rb$ to be the set of convex functions $\phi \colon N_\bR \to \bR$ such that
\begin{align}\label{eq:P-delta}
\sup_{n \in N_\bR} \left| \phi(n)-\phi_h(n) \right| < \infty,
\end{align}
where $\phi_{h} \colon N_\bR \to \bR$ is the support function of 
$\Delta_h:=\lc m \in M_\bR \relmid \la m, n \ra \geq -h(n), \forall n \in N_\bR \rc$ defined by
\begin{align}
    \phi_{h}(n):=-\inf_{m \in \Delta_h \cap M} \la m, n \ra.
\end{align}
There is a bijection between $\scP\lb \Delta_h \rb$ and the space of semi-positive toric metrics on $L_h^{\an}$ by \cite[Theorem 4.8.1(1)]{MR3222615}.
The main result of this article is the following: 

\begin{theorem}\label{th:main}
The following hold:
\begin{enumerate}
    \item 
The following are equivalent:
\begin{enumerate}
\item[(i)] There exists a convex function $\phi \in \scP\lb \Delta_h \rb$
such that, for any facet $\tau$ of $B_{\check{h}}$, the convex function $\phi$ satisfies the real Monge--Amp\`{e}re equation
    \begin{align}\label{eq:real-MA}
        \MA_\bR \lb \left. \phi \right|_{\rint (\tau)}\rb
        = \mu \lb A_h \rb \tilde{\nu}|_{\rint (\tau)},
    \end{align}
    where $\rint (\tau)$ is the relative interior of $\tau$ and
    $\MA_\bR$ is the real Monge--Amp\`{e}re operator defined for convex functions on $\rint (\tau)$ (cf.~\cite[Section 2.5]{MR4688155}).
     \item[(ii)] The solution of the non-archimedean Monge--Amp\`{e}re equation \eqref{eq:NA-MA} is given by the semi-positive toric metric on $L_h^{\an}$ induced by a convex function $\phi \in \scP\lb \Delta_h \rb$.
\end{enumerate}
\item If either (i) or (ii) holds, then the NA MA-real MA comparison property holds for $(X^{\an},L_h^{\an})$, which implies the metric SYZ conjecture (\pref{cj:syz}) for the toric degeneration 
$\scX |_{\mathbb{D}^*}$.
\end{enumerate}
\end{theorem}

There are some conditions that we impose for \pref{th:main} but do not mention here, such as \pref{cd:add} and \pref{cd:add2}.
They will be stated precisely in \pref{sc:setting}.

\pref{th:main} is a generalization of a part of \cite{MR4692438, MR4701493}.
Let $\varphi \colon N_\bR \to \bR, \check{\varphi} \colon M_\bR \to \bR$ be the convex piecewise linear functions that correspond to the anti-canonical sheaves on the toric varieties associated with the fans $\Sigma, \Sigmav$ respectively.
In \cite[Section 3]{MR4701493}, \pref{th:main} is shown under the additional assumption that $r=1$, the reflexive polytope $\Delta$ is Delzant, and $\check{h}=\check{\varphi}$.
The same thing is proved earlier in \cite{MR4692438}, supposing further that the ambient smooth toric variety is the projective space.

By applying \cite[Theorem 2.31]{MR4701493}, we obtain the following generalization of \cite[Theorem 1.2]{MR4701493}:

\begin{corollary}\label{cr:Li}
    The metric SYZ conjecture holds for the family $\lc \scX_t\rc_t$ under the assumption that
    \begin{itemize}
    \item $r=1$ and $\Delta$ is a reflexive Delzant polytope (this implies $\Sigma'=\Sigma$),
    \item $\Sigmav'=\Sigmav$,
    \item \pref{cd:add2} stated in \pref{sc:setting} holds, and
    \item the condition $1$ or $2$ stated in \cite[Theorem 2.31]{MR4701493} is satisfied.
    \end{itemize}
\end{corollary}
Similarly, by applying \cite[Theorem 1]{AH23}, we obtain the following generalization of \cite[Corollary 1]{AH23}:

\begin{corollary}\label{cr:AH}
    The metric SYZ conjecture holds for the family $\lc \scX_t\rc_t$ under the assumption that
    \begin{itemize}
    \item $r=1$ and $\Delta$ is a reflexive Delzant polytope (this implies $\Sigma'=\Sigma$),
    \item $h=\varphi$, i.e., the polarization is anticanonical, and
    \item $(\Delta, \check{h})$ is stable in the sense of \cite[Definition 1, Remark 2]{AH23}.
    \end{itemize}
\end{corollary}

The condition of being stable of \cite{AH23} is formulated in terms of optimal transport plans in the case of hypersurfaces, and is equivalent to the existence of a solution to the real Monge--Amp\`{e}re equation \cite[Theorem 1]{AH23}.
It has also been proved that some examples that have enough symmetries such as the reflexive simplex, cube, and Weyl polytopes actually satisfy the stability condition \cite{AH23, DH24}.
For some other sufficient or necessary conditions for the existence of a solution to the real Monge--Amp\`{e}re equation, we refer the reader to \cite{MR4701493} and \cite[Section 9.2]{AH23}.

\begin{remark}
In \pref{cr:AH}, we do not need to impose \pref{cd:add2} in contrast to \pref{cr:Li}, since $h=\varphi$, and the function $\tilde{\varphi}$ of \eqref{eq:tvp} plays the same role as the function $\tilde{h}$ of \pref{cd:add2}.
\end{remark}

Notice that all the preceding works such as \cite{MR4692438, MR4701493, AH23, DH24} proved \pref{cr:Li} and \pref{cr:AH} under the additional assumption $\check{h}=\check{\varphi}$.

\subsection{Strategy of the proof}

The proof of \pref{th:main} is basically based on the same strategy as \cite{MR4692438} and \cite{MR4701493}.
We compare the non-archimedean/real Monge--Amp\`{e}re measures by using the result of \cite{MR4204708} to show \pref{th:main}(1).
For proving \pref{th:main}(2), i.e., for confirming the NA MA-real MA comparison property, we need a model satisfying the required condition.
Although toric degenerations may not be snc in general,
in the case where the special fiber of a toric degeneration is the toric boundary of the ambient smooth toric variety as in \cite{MR4692438, MR4701493}, we are able to use the toric degeneration itself as the model, 
since it is snc at all the $0$-dimensional strata of the special fiber (cf.~\cite[Remark 3.5]{MR4701493}).
However, this is not the case in the setup of this article.
For this issue, in this article, we instead use a minimal snc-model constructed in \cite{Yam24} by blowing up the toric degeneration along irreducible components of the special fiber repeatedly.

\subsection{Organization of this article}

In \pref{sc:setting}, we recall the construction by Gross \cite{MR2198802} of toric degenerations of Calabi--Yau complete intersections, and state the precise setup of \pref{th:main}.
In \pref{sc:sphere}, we recall the definition of tropical spheres of \cite{MR2187503}.
In \pref{sc:proof2} and \pref{sc:proof1},
we give a proof of \pref{th:main} (2) and (1) respectively.

\section{Toric degenerations of Calabi--Yau complete intersections}\label{sc:setting}

Regarding the sign convention, we follow \cite{MR2198802}.
In particular, the tropicalization map is defined by \eqref{eq:trop} as opposed to \cite{MR4692438, MR4701493, AH23}, and convex functions considered in this article are all lower convex as in loc.cit.

Let $d$ and $r$ be positive integers, and $M$ be a free $\bZ$-module of rank $d+r$.
We set $N:=\Hom(M, \bZ), M_\bR:=M \otimes_\bZ \bR$, and $N_\bR:=N \otimes_\bZ \bR$.
Let $\Delta \subset M_\bR$ be a reflexive polytope, and $\Delta^\ast := \lc n \in N_\bR \relmid \la \Delta, n \ra \geq -1 \rc$ be its polar polytope.
Let further $\Sigma \subset N_\bR$ be the normal fan of $\Delta$.
Consider the convex piecewise linear function $\varphi \colon N_\bR \to \bR$ that corresponds to the anti-canonical sheaf on the toric variety associated with the fan $\Sigma$.
Let $\lc e_1, \cdots, e_l \rc$ be the set of primitive generators of one-dimensional cones of the fan $\Sigma$.
Then we have $\varphi(e_i)=1$ for any $i \in \lc 1, \cdots, l \rc$.
A Minkowski decomposition $\Delta =\Delta_1+ \cdots + \Delta_r$ is called a \emph{nef-partition} if the induced decomposition $\varphi = \varphi_1 + \cdots + \varphi_r$ satisfies $\varphi_i(e_j) \in \lc 0, 1\rc$ for any $i \in \lc 1, \cdots, r\rc$ and $j \in \lc 1, \cdots, l\rc$.
Let $\nabla_i \subset N_\bR$ be the convex hull of $0 \in N_\bR$ and all $e_j$ such that $\varphi_i(e_j) =1$.
Then the Minkowski sum $\nabla:=\nabla_1 + \cdots + \nabla_r$ becomes a reflexive polytope.
We also write its polar polytope as $\nabla^\ast \subset M_\bR$.
Let $\Sigmav \subset M_\bR$ be the normal fan of $\nabla$, and $\varphiv \colon M_\bR \to \bR$ be the piecewise linear function corresponding to the anti-canonical sheaf on the toric variety associated with the fan $\Sigmav$.
We also write the decomposition of $\varphiv$ induced by $\nabla=\nabla_1 + \cdots + \nabla_r$ as $\varphiv = \varphiv_1 + \cdots + \varphiv_r$.
For a lattice polytope $\mu \subset \partial \Delta^\ast$, we define
\begin{align}
\beta_i^\ast(\mu):= \lc n \in \mu \relmid \varphi_i(n)=1 \rc \subset \nabla_i \cap \partial \Delta^\ast, \quad \rotatebox[origin=c]{180}{$\beta$} (\mu):=\sum_{i=1}^r \beta_i^\ast(\mu) \subset \nabla.
\end{align}
The subset $\beta_i^\ast(\mu)$ is a face of the lattice polytope $\mu$.

We take coherent subdivisions $\Sigma' \subset N_\bR, \Sigmav' \subset M_\bR$ of the fans $\Sigma, \Sigmav$ whose fan polytopes (i.e., the convex hulls of primitive generators of all one-dimensional cones) remain to be $\Delta^\ast, \nabla^\ast$ respectively, and consider the integral piecewise linear functions $h, \check{h}$ of \eqref{eq:hh}.
We assume that the functions 
\begin{align}
h':=h-\varphi \colon N_\bR \to \bR\\
\check{h}':=\check{h}-\varphiv \colon M_\bR \to \bR
\end{align}
are convex (not necessarily strict convex) on $\Sigma'$ and $\Sigmav'$ respectively.
Let
\begin{align}
\Delta_h&:=\lc m \in M_\bR \relmid \la m, n \ra \geq -h(n), \forall n \in N_\bR \rc \\
\nabla_{\check{h}}&:=\lc n \in N_\bR \relmid \la m, n \ra \geq -\check{h}(m), \forall m \in M_\bR \rc \\
\nabla_{\check{h}'}&:=\lc n \in N_\bR \relmid \la m, n \ra \geq -\check{h}'(m), \forall m \in M_\bR \rc
\end{align}
be the Newton polytopes of $h, \check{h}$, and $\check{h}'$ respectively.
We set
\begin{align}
\scrR^{\check{h}}_{\Delta^\ast}&:= \lc (F_1, F_2) \relmid F_1 \prec \Delta^\ast,  F_2 \prec \nabla^{\check{h}'}\ \mathrm{such\ that}\ \rotatebox[origin=c]{180}{$\beta$} (F_1) \neq \emptyset, F_1+F_2 \prec \Delta^\ast + \nabla^{\check{h}'}\rc,\\
\scrP^{\check{h}}_{\Delta^\ast}&:= \lc \rotatebox[origin=c]{180}{$\beta$} (F_1)+F_2 \relmid (F_1, F_2) \in \scrR^{\check{h}}_{\Delta^\ast} \rc, \\
B^{\check{h}}_\nabla&:=\bigcup_{(F_1, F_2) \in \scrR^{\check{h}}_{\Delta^\ast}} \rotatebox[origin=c]{180}{$\beta$} (F_1)+F_2 \subset N_\bR.
\end{align}
One has $B^{\check{h}}_\nabla \subset \partial \nabla_{\check{h}}$ (\cite[Corollary 3.3]{MR2198802}).
We also consider
\begin{align}\label{eq:t-Delta-i}
\tilde{\Delta}_i:=\lc (m, l) \in M_\bR \oplus \bR \relmid m \in \Delta_i, l \geq \check{h}'(m) \rc, \quad 
\tilde{\Delta}:=\sum_{i=1}^r \tilde{\Delta}_i
\end{align}
and the normal fan $\tilde{\Sigma} \subset N_\bR \oplus \bR$ of $\tilde{\Delta}$.
The fan $\tilde{\Sigma}$ is the union of
\begin{enumerate}
\item $\lc \cone \lb F_1 \rb \times \lc 0 \rc \relmid F_1 \prec \Delta^\ast \rc$,
\item $\lc \cone(F_1) \times \lc 0 \rc + \cone \lb F_2 \times \lc 1 \rc \rb \relmid F_1 \prec \Delta^\ast,  F_2 \prec \nabla^{\check{h}'}, F_1+F_2 \prec \Delta^\ast + \nabla^{\check{h}'} \rc$, and 
\item $\lc \cone \lb F_2 \times \lc 1 \rc \rb \relmid F_2 \prec \nabla^{\check{h}'} \rc$
\end{enumerate}
(\cite[Proposition 3.7]{MR2198802}).
The polytope $\tilde{\Delta}$ determines a line bundle on the toric variety associated with $\tilde{\Sigma}$.
It is induced by the piecewise linear function $\tilde{\varphi}$ on $\tilde{\Sigma}$ defined by
\begin{align}\label{eq:tvp}
\tilde{\varphi} (\tilde{n}):=-\inf_{\tilde{m}\in \tilde{\Delta}} \la \tilde{m},\tilde{n}\ra,
\end{align}
where $\tilde{n} \in N_\bR \oplus \bR$.
The Minkowski decomposition $\tilde{\Delta}=\sum_{i=1}^r \tilde{\Delta}_i$ induces a decomposition $\tilde{\varphi}=\sum_{i=1}^r \tilde{\varphi}_i$ with $\tilde{\varphi}_i \lb \lb n , 0 \rb \rb=\varphi_i(n)$ and $\tilde{\varphi}_i \lb \lb n , 1 \rb \rb=0$ for any $n \in \nabla^{\check{h}'}$.

We take a subdivision $\tilde{\Sigma}'$ of the fan $\tilde{\Sigma} \subset N_\bR \oplus \bR$ so that it satisfies the following condition:
\begin{condition}\label{cd:orig}
The following hold:
\begin{enumerate}
\item The fan $\lc C \cap (N_\bR \oplus \lc 0\rc) \relmid C \in \tilde{\Sigma}' \rc$ coincides with the fan $\Sigma'$.
\item Every $1$-dimensional cone of $\tilde{\Sigma}'$ not contained in $N_\bR \oplus \lc 0\rc$ is generated by a primitive vector $(n, 1)$ with $n \in \nabla^{\check{h}'} \cap N$.
\end{enumerate}
\end{condition}
Such a subdivision $\tilde{\Sigma}'$ is called \emph{good} in \cite[Definition 3.8]{MR2198802}.
We suppose that one can take a good subdivision $\tilde{\Sigma}'$ that also satisfies the following:
\begin{condition}\label{cd:add}
The fan $\tilde{\Sigma}'$ is unimodular, i.e., all the cones are generated by a subset of a basis of the lattice $N \oplus \bZ \subset N_\bR \oplus \bR$.
(In particular, the subfan $\Sigma' \subset \tilde{\Sigma}'$ is required to be unimodular.)
\end{condition}
We also impose the following:

\begin{condition}\label{cd:add2}
One can take a strictly convex integral piecewise linear function 
$\tilde{h} \colon N_\bR \oplus \bR \to \bR$ on $\tilde{\Sigma}'$ whose restriction to $N_{\bR} \oplus \lc 0 \rc$ coincides with $h$.
(Regarding the existence of such a function $\tilde{h}$, we refer the reader to \cite[Definition 3.17, Theorem 3.26]{MR2198802}.)
\end{condition}

We fix such a subdivision $\tilde{\Sigma}'$ and a function $\tilde{h}$ in the following.
The fan $\tilde{\Sigma}'$ consists of the following three sorts of cones (\cite[Observation 3.9]{MR2198802}):
\begin{enumerate}
\item cones of the form $\cone(\mu) \times \lc 0 \rc$ for $\cone(\mu) \in \Sigma'$ with $\mu \subset \partial \Delta^\ast$,
\item cones of the form $\cone(\mu) \times \lc 0 \rc + \cone \lb \nu \times \lc 1 \rc \rb$ where $\cone(\mu) \in \Sigma'$ for $\mu \subset \partial \Delta^\ast,  \nu \subset \partial \nabla^{\check{h}'}$, and $\mu + \nu$ is contained in a face of $\Delta^\ast + \nabla^{\check{h}'}$, and
\item cones contained in $\cone(\nabla^{\check{h}'} \times \lc 1\rc)$.
\end{enumerate}
Cones of the second type with $\rotatebox[origin=c]{180}{$\beta$}(\mu) \neq \emptyset$ are called \textit{relevant}.
We set
\begin{align}
\scrR(\tilde{\Sigma}')&:=\lc (\mu, \nu) \relmid 
\begin{array}{l}
\mu \subset \partial \Delta^\ast, \nu \subset \partial \nabla^{\check{h}'} \\
\cone(\mu) \times \lc 0 \rc + \cone \lb \nu \times \lc 1 \rc \rb \mathrm{\ is\ a\ relevant\ cone\ in\ } \tilde{\Sigma}'
\end{array}
\rc, \\
\scrP(\tilde{\Sigma}')&:=\lc \rotatebox[origin=c]{180}{$\beta$} (\mu) + \nu \relmid (\mu, \nu) \in \scrR(\tilde{\Sigma}') \rc.
\end{align}
Then one has
\begin{align}
B^{\check{h}}_\nabla:=\bigcup_{(\mu, \nu) \in \scrR(\tilde{\Sigma}')} \rotatebox[origin=c]{180}{$\beta$} (\mu) + \nu,
\end{align}
and $\scrP(\tilde{\Sigma}')$ is a polyhedral decomposition of $B^{\check{h}}_\nabla$ \cite[Proposition 3.12, Definition 3.13]{MR2198802}.
In \cite{MR2198802}, we also define an integral affine structure with singularities on $B^{\check{h}}_\nabla$, although we do not use it in this article.

Let $X_{\tilde{\Sigma}'}$ be the toric variety over $\bC$ associated with the fan $\tilde{\Sigma}'$, and $t$ be the regular function on $X_{\tilde{\Sigma}'}$ corresponding to $(0, 1) \in M \oplus \bZ$, which defines the morphism $f \colon X_{\tilde{\Sigma}'} \to \bA^1$.
Let further $\scL_i$ be the line bundle on $X_{\tilde{\Sigma}'}$ associated with the polytope $\tilde{\Delta}_i$, and take a general section $s_i \in \Gamma \lb X_{\tilde{\Sigma}'},  \scL_i \rb$ for every $i \in \lc 1, \cdots, r \rc$.
We consider the closed subscheme $\scX \subset X_{\tilde{\Sigma}'}$ defined by
\begin{align}\label{eq:tss}
ts_1+s_1^0= \cdots =ts_r+s_r^0=0,
\end{align}
where $s_i^0 \in \Gamma \lb X_{\tilde{\Sigma}'},  \scL_i \rb$ is the section defined by $(0, 0) \in \tilde{\Delta}_i$.
By restricting the morphism $f \colon X_{\tilde{\Sigma}'} \to \bA^1$ to $\scX$ and taking the base change to $R:=\bC \ldd t \rdd$, we obtain a morphism $\scX \to \Spec R$.
By abuse of notation, we also write it as $f \colon \scX \to \Spec R$.
This is a toric degeneration of Calabi--Yau varieties \cite[Theorem 3.10]{MR2198802}, and $(B^{\check{h}}_\nabla, \scrP(\tilde{\Sigma}'))$ forms its dual intersection complex \cite[Proposition 3.14]{MR2198802}.
We refer the reader to \cite[Section 4]{MR2213573} for the definitions of toric degenerations and their dual intersection complexes.

Let $K:=\bC \lbb t \rbb$ and $f' \colon X \to \Spec K$ be the base change of the toric degeneration $f \colon \scX \to \Spec R$ to $K$.
The scheme $X$ is a subscheme of the toric variety $X_{\Sigma'}$ over $K$ associated with the fan $\Sigma'$.
Since the fan $\Sigma'$ is unimodular by \pref{cd:add}, the ambient toric variety $X_{\Sigma'}$ is smooth.
By Bertini's theorem, one can see that $X$ is also smooth.
We suppose that $X$ is irreducible.
This is satisfied, for instance, when the nef-partition $\Delta =\Delta_1+ \cdots + \Delta_r$ is $2$-independent in the sense of \cite[Definition 3.1]{MR1463173} (cf.~\cite[Theorem 3.3(ii)]{MR1463173}).
We also give $X$ the polarization obtained by restricting the ample line bundle on the toric variety $X_{\Sigma'}$ associated with the strictly convex function $h$ of \eqref{eq:hh}.
By substituting complex numbers to $t$, we obtain a one-parameter family $\lc \scX_t \rc_{t \in \bC^\ast}$ of polarized complex Calabi--Yau manifolds.

\section{Tropical spheres of Haase--Zharkov}\label{sc:sphere}

We keep the same setup as in \pref{sc:setting}.
For faces $F \prec \nabla_{\check{h}}$, $G \prec \Delta_{h}$, we set 
\begin{align}
M_i(F)&:=\lc m \in \lb \Delta_i \cap M \rb \setminus \lc 0 \rc \relmid \check{h}(m)+\la m, n \ra =0, \forall n \in F \rc\\
N_i(G)&:=\lc n \in \lb \nabla_i \cap N \rb \setminus \lc 0 \rc \relmid h(n)+\la m, n \ra =0, \forall m \in G \rc
\end{align}
and 
\begin{align}
\scrP_{B} 
&:= \lc F \prec \nabla_{\check{h}} \relmid 
M_i(F) \neq \emptyset, \forall i \in \lc 1, \cdots, r \rc \rc\\
\scrP_{A} 
&:= \lc G \prec \Delta_{h} \relmid 
N_i(G) \neq \emptyset, \forall i \in \lc 1, \cdots, r \rc \rc.
\end{align}
We consider the tropical spheres 
\begin{align}\label{eq:B}
B_{\check{h}}&:= \bigcup_{F \in \scrP_B} F,\\ \label{eq:A}
A_{h}&:= \bigcup_{G \in \scrP_A} G.
\end{align}
of \cite[Section 2.3]{MR2187503}.
We have $B_{\check{h}}=B^{\check{h}}_\nabla$ (\cite[Proposition 5.4]{Yam21}), and this is a subset of the tropicalization of $X$ (\cite[Theorem 1.2(1)]{Yam21}).
(It is actually the union of bounded cells of the tropicalization of $X$. See \cite[Proposition 5.7]{Yam21}.)
It is known that $B_{\check{h}}$ and $A_{h}$ are indeed homeomorphic to a $d$-sphere when our nef-partition is irreducible in the sense of \cite[Definition 5.6]{MR1463173} (cf.~\cite[Theorem 2.8]{MR2187503}).
The polytopal complexes $\scrP_{B}$ and $\scrP_{A}$ consist of $d$-dimensional polytopes and their faces. 
The relative interiors of $d$-dimensional polytopes in $\scrP_{B}$ and $\scrP_{A}$ have natural integral affine structures, which induce the Lebesgue measures on $B_{\check{h}}$ and $A_h$, which will be denoted by $\nu$ and $\mu$ respectively.

\section{Proof of \pref{th:main}(2)}\label{sc:proof2}

As in \pref{sc:main-result}, let $L_h$ be the line bundle on the toric variety $X_{\Sigma'}$ over $K$ associated with the polytope $\Delta_h \subset M_\bR$, and $(X_{\Sigma'}^{\an}, L^{\an}_h)$ be the Berkovich analytification of $(X_{\Sigma'}, L_h)$ over $K$.
Let further $\Phi$ be a semi-positive toric metric on $L^{\an}_h$, and $\phi \in \scP\lb \Delta_h \rb$ be the corresponding convex function that satisfies \eqref{eq:P-delta}.
Take a basis $\lc m_1, \cdots m_{d+r} \rc$ of $M_\bR$ so that $m_i \in \Delta_h \cap M$ for all $i \in \lc 1, \cdots, d+r\rc$.
For any $m \in M_\bR$, we write it as $m=\sum_{i=1}^{d+r} p_i m_i$ with $p_i \in \bR$, and define the metric $\log \left| s^m \right|$ on $L^{\an}_h$ by
\begin{align}
    \log \left| s^m \right|:=\sum_{i=1}^{d+r} p_i \log \left| s^{m_i} \right|,
\end{align}
where $\log \left| s^{m_i} \right|$ is the metric associated with the section of $L_h$ corresponding to $m_i \in \Delta_h \cap M$. 

\begin{lemma}\label{lm:toric-metric}
There exists a bounded continuous function $\phi^\ast \colon \partial\Delta_{h} \to \bR$ such that
\begin{align}\label{eq:toric-metric}
    \Phi=\sup_{m \in \partial \Delta_h} \lc \log |s^m|-\phi^\ast(m) \rc.
\end{align}
\end{lemma}
\begin{proof}
By \cite[Lemma 10]{AH23}, there exists a bounded continuous function $\phi^*:\partial\Delta_{h} \to \bR$ such that 
\begin{align}\label{eq:c-convex}
\phi (n)=\sup_{m \in \partial \Delta_{h}} \lc \la m, -n \ra-\phi^*(m) \rc
\end{align}
for any $n\in N_\bR$. 
The sign of $n$ in \eqref{eq:c-convex} differs from that of \cite{MR4692438, MR4701493, AH23}.
This discrepancy is due to our sign convention mentioned in the beginning of \pref{sc:setting}.
Let $n \in N_\bR$ be an arbitrary element.
In additive notation of metrics, the evaluation of the section $s^{m'}$ with $m'=0$ of $L_h$ by the metric of the right hand side of \eqref{eq:toric-metric} at the Gauss point of $n$ is
\begin{align}
\sup_{m \in \partial \Delta_{h}} \lc -\sum_{i=1}^{d+r} p_i \la m_i, n \ra-\phi^\ast(m) \rc
=\sup_{m \in \partial \Delta_{h}} \lc \la m, -n \ra-\phi^\ast(m) \rc
=\phi (n).
\end{align}
This implies \eqref{eq:toric-metric}, since $\Phi$ is the semi-positive toric metric associated with $\phi$.
\end{proof}

One can construct a minimal snc-model $\scrX \to \Spec R$ of $X$ by blowing up the toric degeneration $\scX \to \Spec R$ along irreducible components of the special fiber repeatedly \cite[Proposition 5.13]{Yam24}.
We refer the reader to \cite[Section 5.2]{Yam24} for more details of the construction of the model $\scrX$.
Let $\Sk(\scrX) \subset X^{\an}$ and $\rho_\scrX \colon X^{\an} \to \Sk(\scrX)$ be the skeleton and the Berkovich retraction associated with the model $\scrX$, where $X^{\an}$ is the Berkovich analytification of $X$ over $K$.
We consider the tropicalization map
\begin{align}\label{eq:trop}
    \trop \colon \bT^{\an} \to N_\bR=\Hom \lb M, \bR \rb, \quad x 
    \mapsto \lb m \mapsto -\log |z^m|_x \rb,
\end{align}
where $\bT^{\an} \subset \lb X_{\Sigma'} \rb^{\an}$ is the Berkovich analytification of $\bT:=\Spec K \ld M \rd$.

\begin{lemma}{\rm(cf.~\cite[Lemma 7.1(2)]{MR4692438})}\label{lm:trop-rho}
One has
    \begin{align}
        \trop=\trop \circ \rho_\scrX
    \end{align}
on the preimages by the retraction $\rho_\scrX$ of the interiors of all the facets of the skeleton $\Sk(\scrX)$.
\end{lemma}
\begin{proof}
    Let $\tau \subset \Sk(\scrX)$ be a facet of the skeleton $\Sk(\scrX)$, and $x \in X^{\an}$ be a point such that $\rho_\scrX(x) \in \rint(\tau)$.
    We will show $\trop(x)=\trop \circ \rho_\scrX(x)$.
    The facet $\tau$ is a simplex having $d+1$ vertices.
    Let $D_i$ $(1 \leq i \leq d+1)$ be the irreducible components of the special fiber of $\scrX$ corresponding to the vertices of $\tau$.
    Since $\rho_\scrX(x) \in \rint(\tau)$, the center of the point $x$ is the intersection point $\bigcap_{i=1}^{d+1} D_i$.
    We can see from the computation of repetitive blow-ups producing the model $\scrX$ in \cite[Section 5.3]{Yam24} that there is a local chart $\scrU \subset \scrX$ which contains the point $\bigcap_{i=1}^{d+1} D_i$ and satisfies the following:
    \begin{enumerate}
    \item The chart $\scrU$ is written as
    \begin{align}
         \scrU=\Spec \left. R \ld z^{m_l} : 1 \leq l \leq d+r+1 \rd \middle/ \lb t-\prod_{l=1}^{d+1} z^{m_l}, -z^{m_{d+1+i}} +f_i \cdot \prod_{l = l_i}^{d+1} z^{m_l} (1 \leq i \leq r) \rb \right.,
    \end{align}
    where
    \begin{itemize}
    \item $\lc m_l \relmid 1 \leq l \leq d+r+1 \rc$ is a basis of $M \oplus \bZ$,
    \item $f_i \in \bC \ld z^{m_l} : 1 \leq l \leq d+r+1 \rd$ $(1 \leq i \leq r)$ are polynomials that are proportional to the general sections $s_i$ that we took in \eqref{eq:tss}
    (in particular, their constant terms are not zero), and
    \item $l_i$ $(1 \leq i \leq r)$ are integers such that $1 \leq l_i \leq d+1$.
\end{itemize}
\item On the chart $\scrU$, the irreducible component $D_i$ $(1 \leq i \leq d+1)$ is defined by $z^{m_i}$
\end{enumerate}
    (cf.~\cite[Lemma 5.2, Lemma 5.3, Lemma 5.7(2), Corollary 5.8, Lemma 5.12]{Yam24}).
    Since $\rho_\scrX(x) \in \rint(\tau)$, we have
\begin{align}\label{eq:logz}
    -\log |z^{m_l}|_{x} &> 0\quad (1 \leq l \leq d+1),\\
    -\log |z^{m_l}|_{x} &\geq 0 \quad (d+2 \leq l \leq d+r+1).
\end{align}
From these, one can get $-\log | f_i |_x \geq 0$, and 
\begin{align}\label{eq:logz2}
    -\log |z^{m_l}|_{x}
=-\log \left| f_i \cdot \prod_{l = l_i}^{d+1} z^{m_l} \right|_x
=-\log | f_i |_x -\sum_{l =l_i}^{d+1} \log |z^{m_l}|_{x}>0
\end{align}
for $l=d+1+i$ $(1 \leq i \leq r)$.
By using \eqref{eq:logz} and \eqref{eq:logz2}, one can get
\begin{align}\label{eq:logf}
    -\log |f_i|_{x} =0,
\end{align}
since the constant term of the polynomial $f_i$ is not zero.

    We consider the maps
    \begin{align}
        M \oplus \bZ &\to \bR, \quad m \mapsto -\log |z^m|_x, \\
        M \oplus \bZ &\to \bR, \quad m \mapsto -\log |z^m|_{\rho_{\scrX}(x)}.
    \end{align}
    These maps define the elements $\trop(x) \times \lc 1 \rc$ and $\trop\lb \rho_\scrX (x) \rb \times \lc 1 \rc$ in $N_\bR \oplus \bR$ respectively.
    Hence, it suffices to show 
    \begin{align}\label{eq:ttr}
        -\log |z^{m_l}|_x=-\log |z^{m_l}|_{\rho_{\scrX}(x)}
    \end{align}    
    for all $l \in \lc 1, \cdots, d+r+1\rc$.
    If $1 \leq l \leq d+1$, then \eqref{eq:ttr} is obvious from the definition of the Berkovich retraction.
If $l=d+1+i$ $(1 \leq i \leq r)$, then
\begin{align}
\log |z^{m_l}|_{\rho_{\scrX}(x)}
&=\log \left| f_i \cdot \prod_{l = l_i}^{d+1} z^{m_l} \right|_{\rho_{\scrX}(x)}
=\sum_{l =l_i}^{d+1} \log |z^{m_l}|_{x},\\
\log |z^{m_l}|_{x}
&=\log \left| f_i \cdot \prod_{l = l_i}^{d+1} z^{m_l} \right|_x
=\sum_{l =l_i}^{d+1} \log |z^{m_l}|_{x},
\end{align}
where we use \eqref{eq:logf} for the last equality.
Thus we get \eqref{eq:ttr} also for $l=d+1+i$.
We conclude the claim.
\end{proof}

The strictly convex function $\tilde{h}$ of \pref{cd:add2} gives a line bundle on the toric variety $X_{\tilde{\Sigma}'}$.
By taking its base change to $R$, we obtain a model of the line bundle $L_h$, which will be denoted by $\scL$.
By abuse of notation, let $L_h$ (resp. $\scL$) also denote the restriction of $L_h$ (resp. $\scL$) to $X$ (resp. $\scX$).
Let further $\scrL \to \scrX$ denote the pull-back of the model $\scL$ of $L_h$ via the composition of the repetitive blow-ups $\scrX \to \scX$.
We also write the restriction of the semi-positive toric metric $\Phi$ to the analytification of $L_h \to X$ as $\Phi$.
We define the potential function $\psi \colon X^{\an} \to \bR$ of $\Phi$ by
\begin{align}
    \psi:=\Phi-\phi_\scrL,
\end{align}
where $\phi_\scrL$ is the model metric associated with $\scrL$.

By \cite[Theorem 3.3.3]{MR3595497}, the skeleton $\Sk (\scrX)$ associated with the minimal snc-model $\scrX$ equals the essential skeleton $\Sk(X)$ of $X$.
\cite[Theorem 1.1(1)]{Yam24} claims that the tropicalization map $\trop$ of \eqref{eq:trop} gives an identification of $\Sk(X)$ with the dual intersection complex $B^{\check{h}}_\nabla$ of the toric degeneration $f \colon \scX \to \Spec R$.
The simplicial structure of $\Sk(\scrX)$ is a subdivision of the polyhedral structure $\scrP(\tilde{\Sigma}')$ of $B^{\check{h}}_\nabla$.
In the following, we identify $\Sk(\scrX)$ with $B^{\check{h}}_\nabla$, and think of the restriction of the convex function $\phi \in \scP\lb \Delta_h \rb$ to $B_\nabla^{\check{h}}=B_{\check{h}}$ also as a function on $\Sk(\scrX)$.

\begin{lemma}{\rm(cf.~\cite[Lemma 3.6]{MR4701493})}\label{lm:potential-retraction}
The following hold:
\begin{enumerate}
\item
One has
\begin{align}
    \psi=\psi \circ \rho_\scrX
\end{align}
on the preimages by $\rho_\scrX$ of the interiors of all the facets of the skeleton $\Sk(\scrX)$.
\item For any facet $\tau$ of $\Sk(\scrX)$, the function 
\begin{align}
    \left. \psi\right|_\tau-\left. \phi \right|_\tau
\end{align}
is an integral affine function on $\tau$. 
\end{enumerate}
\end{lemma}
\begin{proof}
First, we show the former claim.
We can see from \pref{lm:trop-rho} that it suffices to show that for any maximal dimensional polytope $\sigma=\rotatebox[origin=c]{180}{$\beta$} (\mu) + \nu \in \scrP(\tilde{\Sigma}')$, the value of the potential function $\psi$ depends only on the image of the tropicalization map $\trop$ over $\trop^{-1}\lb \sigma \rb$.
Let $x \in \bT^{\an} \cap X^{\an}$ be a point such that $\trop(x) \in \sigma$.
We consider the element
\begin{align}\label{eq:trop1}
\lb M \oplus \bZ \ni \tilde{m} \mapsto -\log |z^{\tilde{m}} |_x \rb \in \Hom \lb M \oplus \bZ, \bR \rb =N_\bR \oplus \bR,
\end{align}
which is equal to $\trop(x) \times \lc 1\rc \in N_\bR \oplus \bR$.
Since $\trop(x) \in \sigma=\rotatebox[origin=c]{180}{$\beta$} (\mu) + \nu$, we have
\begin{align}
    \trop(x) \times \lc 1\rc \in C_{\sigma} \cap \lb N_\bR \times \lc 1 \rc \rb,
\end{align}
where 
\begin{align}\label{eq:cone}
    C_\sigma:=\cone(\mu) \times \lc 0 \rc + \cone \lb \nu \times \lc 1 \rc \rb \in \tilde{\Sigma}'.
\end{align}
This implies that the point $x$ is contained in
\begin{align}\label{eq:bet}
    \lc y \in \lb \Spec A_\sigma \times_{k[t]} K \rb^{\an} \cap X^{\an} \relmid -\log |f|_y  \geq 0, \forall f \in A_\sigma \otimes_{k[t]} R \rc
\end{align}
with $A_\sigma:=k \ld \check{C}_\sigma \cap (M \oplus \bZ ) \rd$, where $\check{C}_\sigma$ is the cone dual to $C_\sigma$.

The model metric $\phi_\scrL$ coincides with the model metric $\phi_\scL$ associated with $\scL$ by \cite[Lemma 5.10(ii)]{MR4192993}.
By the definition of model metrics, the evaluation of a trivializing section of $\scL$ on $\Spec A_\sigma \times_{k[t]} R$ by the model metric $\phi_\scL$ is constantly $0$ (in additive notation) on \eqref{eq:bet}.
Let $\lb m_\sigma, l_\sigma \rb \in M \oplus \bZ$ be the lattice point that corresponds to a section of $\scL$ trivializing $\scL$ on $\Spec A_\sigma \times_{k[t]} R$.
Then using \pref{lm:toric-metric}, we obtain
\begin{align}
\psi(x)&=\lb \Phi-\phi_\scL\rb(x)\\
&=\sup_{m \in \partial \Delta_h} \lc \log \left| \frac{z^m}{t^{l_\sigma} z^{m_\sigma}} \right|_x-\phi^\ast(m) \rc\\
    &=
    \la m_\sigma, \trop(x) \ra+l_\sigma
    +\sup_{m \in \partial \Delta_h} \lc \la m, -\trop (x) \ra -\phi^\ast(m) \rc\\
    &=\la m_\sigma, \trop(x) \ra+l_\sigma+\phi \lb \trop (x) \rb.
\end{align}
It is clear that this depends only on the image by $\trop$.
Thus we can conclude the former claim.
The latter claim also follows from the last equation.
Notice that the element $\lb m_\sigma, l_\sigma \rb$ depends only on $\sigma$.
\end{proof}
In particular, if (ii) of \pref{th:main} holds and $\Phi$ is the solution to \eqref{eq:NA-MA}, then \pref{lm:potential-retraction}(1) confirms the NA MA-real MA comparison property of \cite[Definition 3.3]{MR4688155}, that is, \pref{th:main}(2).

\begin{remark}
For the NA MA-real MA comparison property in the case of Calabi--Yau hypersurfaces, see also \cite[Corollary 8.2]{MR4692438} and \cite[Theorem 5.4]{PS22}.
\end{remark}

\section{Proof of \pref{th:main}(1)}\label{sc:proof1}

In the following, let $X_{\Sigma'}$ denote the toric variety over $\bC$ associated with the fan $\Sigma'$.
For each $i \in \lc 1, \cdots, r \rc$, let $\scX_t^i$ denote the complex hypersurface in the toric variety $X_{\Sigma'}$, which is defined by the equation
\begin{align}
    ts_i+s_i^0=0 \quad (t \in \bC)
\end{align}
appearing in \eqref{eq:tss}.
For any subset $I \subset \lc 1, \cdots, r \rc$ and any line bundle $L$ on the toric variety $X_{\Sigma'}$, we define
\begin{align}
    h(I, L)
    :=
    \sum_{J \subset I} (-1)^{|J|} \cdot h^0 \lb X_{\Sigma'}, L \otimes \scO \lb-\sum_{j \in J} \scX_t^j \rb \rb.
\end{align}

\begin{lemma}\label{lm:claim-dim}
Let $I \subset \lc 1, \cdots, r \rc$ be a non-empty subset, and $L'$ be any line bundle  on the toric variety $X_{\Sigma'}$.
    For a sufficiently large $k \in \bZ_{>0}$, we have
    \begin{align}\label{eq:claim}
        h^0 \lb \bigcap_{i\in I}\scX_t^i, kL_h \otimes L' \rb
        =
        h \lb I, kL_h \otimes L' \rb,
    \end{align}
    where $kL_h$ is the line bundle on the toric variety $X_{\Sigma'}$ associated with the polytope $k \Delta_h \subset M_\bR$.
\end{lemma}
\begin{proof}
We prove the lemma by induction on $|I|$.
Let $I_0 \subsetneq \lc 1, \cdots, r \rc$ be a subset, and $i_0 \in \lc 1, \cdots, r \rc \setminus I_0$ be an arbitrary element.
We set $I_1:=I_0 \cup \lc i_0 \rc$.
Consider the short exact sequence 
\begin{align}
    0 \to kL_h \otimes L' \otimes \scO \lb -\scX_t^{i_0}\rb
    \to kL_h \otimes L'
    \to \left. kL_h \otimes L' \right|_{\bigcap_{i \in I_1} \scX_t^{i}}
    \to 0
\end{align}
on $\bigcap_{i \in I_0} \scX_t^{i}$.
(We set $\bigcap_{i\in I_0}\scX_t^i:=X_{\Sigma'}$ for $I_0=\emptyset$.)
By this and Serre's vanishing theorem
\begin{align}
    h^1 \lb \bigcap_{i \in I_0} \scX_t^{i}, kL_h \otimes L' \otimes \scO \lb -\scX_t^{i_0}\rb \rb=0\quad (k \gg 1),
\end{align}
we get
\begin{align}\label{eq:exact-serre}
h^0 \lb \bigcap_{i \in I_1} \scX_t^{i}, kL_h \otimes L' \rb
=
    h^0 \lb \bigcap_{i \in I_0} \scX_t^{i}, kL_h \otimes L' \rb
    -
    h^0 \lb \bigcap_{i \in I_0} \scX_t^{i}, kL_h \otimes L' \otimes \scO \lb -\scX_t^{i_0} \rb \rb.
    \end{align}
When $I_0=\emptyset$, this implies \eqref{eq:claim} for the case where $|I|=1$.

In the following, supposing that \pref{lm:claim-dim} holds when $|I| = l$ $(l \in \bZ_{\geq 1})$, we prove that it holds also when $|I| = l+1$.
For $I \subset \lc 1, \cdots, r \rc$ such that $|I| = l+1$, we choose an arbitrary element $i_0 \in I$, and set $I_0:=I \setminus \lc i_0 \rc$ and $I_1:=I=I_0 \cup \lc i_0 \rc$.
By \eqref{eq:exact-serre} and the induction hypothesis, we get
\begin{align}\label{eq:h0}
h^0 \lb \bigcap_{i \in I_1} \scX_t^{i}, kL_h \otimes L' \rb
=
h \lb I_0, kL_h \otimes L' \rb
-
h \lb I_0, kL_h \otimes L' \otimes \scO \lb -\scX_t^{i_0} \rb \rb.
\end{align}
On the other hand, by definition, we have
\begin{equation}\label{eq:I1}
  \begin{split} 
h \lb I_1, kL_h \otimes L' \rb&=\sum_{i_0 \nin J \subset I_1} (-1)^{|J|} \cdot h^0 \lb X_{\Sigma'}, kL_h \otimes L' \otimes \scO \lb-\sum_{j \in J} \scX_t^j \rb \rb\\
&+
\sum_{i_0 \in J \subset I_1} (-1)^{|J|} \cdot h^0 \lb X_{\Sigma'}, kL_h \otimes L' \otimes \scO \lb-\sum_{j \in J} \scX_t^j \rb \rb.
  \end{split}   
\end{equation}
By mapping $J\in \{J\subset I_1 \ |\ i_0 \nin J\}$ to $J\subset I_0$ and $J\in \{J\subset I_1 \ |\ i_0 \in J\}$ to $J\setminus \{i_0\} \subset I_0$,
the former (resp. latter) part of the right hand side of \eqref{eq:I1} equals the former (resp. latter) part of the right hand side of \eqref{eq:h0}. 
Hence, \eqref{eq:h0} is equal to $h \lb I_1, kL_h \otimes L' \rb$.
Thus we conclude \pref{lm:claim-dim}.
\end{proof}

\begin{lemma}{\rm(cf.~\cite[Claim 3.8]{MR4701493})}\label{lm:vol}
    One has
    \begin{align}\label{eq:vol-L}
        d!\cdot \mu(A_h)=( L_h^d ).
    \end{align}
\end{lemma}
\begin{proof}
In \pref{lm:claim-dim}, if $I =\lc 1, \cdots, r \rc$ and $L'$ is the trivial line bundle, then \eqref{eq:claim} is
\begin{align}
    h^0 \lb \scX_t, kL_h \rb
        &=
        \sum_{J \subset \lc 1, \cdots, r \rc} (-1)^{|J|} \cdot h^0 \lb X_{\Sigma'}, kL_h \otimes \scO \lb-\sum_{j \in J} \scX_t^j \rb\rb\\ \label{eq:h0kL}
        &=h^0 \lb X_{\Sigma'}, kL_h \rb
        -
        \sum_{\emptyset \neq J \subset \lc 1, \cdots, r \rc} (-1)^{|J|-1} \cdot h^0 \lb X_{\Sigma'}, kL_h \otimes \scO \lb-\sum_{j \in J} \scX_t^j \rb\rb.
\end{align}
Since the fan $\Sigma'$ is unimodular by \pref{cd:add} and the fan polytope of $\Sigma'$ is $\Delta^\ast$, the set of the integral generators of all the $1$-dimensional cones in $\Sigma'$ coincides with the set $\partial \Delta^\ast \cap N$.
For each $n \in \partial \Delta^\ast \cap N$, we write the toric divisor on $X_{\Sigma'}$ corresponding to the cone $\bR_{\geq 0} \cdot n \in \Sigma'$ as $D_n$.
Then for $j \in \lc 1, \cdots, r \rc$, the divisor $\scX_t^j$ is linearly equivalent to 
\begin{align}\label{eq:divisor}
    \sum_{n \in \partial \Delta^\ast \cap N} \varphi_j (n) \cdot D_n,
\end{align}
where $\varphi_j \colon N_\bR \to \bR$ is the piecewise linear function corresponding to the polytope $\Delta_i$, which appeared in the beginning of \pref{sc:setting}.
We have $\varphi=\sum_{i=1}^r \varphi_i=1$ on $\partial \Delta^\ast$, and $ \varphi_i \geq 0$ on $N_\bR$.
We also know
\begin{align}
\nabla_i \cap \partial \Delta^\ast&=\lc n \in \partial \Delta^\ast \relmid \varphi_i(n)=1 \rc\\
\partial \Delta^\ast \cap N&=\bigsqcup_{i=1}^r \lb \nabla_i \cap N \setminus \lc 0 \rc \rb
\end{align}
\cite[Proposition 3.19, Corollary 3.23]{MR2405763}.
From these, we can see that \eqref{eq:divisor} is equal to
\begin{align}\label{eq:divisor2}
    \sum_{n \in \nabla_j \cap N \setminus \lc 0 \rc} D_n,
\end{align}
and
\begin{align}\label{eq:h0kLX}
    H^0 \lb X_{\Sigma'}, kL_h \otimes \scO \lb-\sum_{j \in J} \scX_t^j \rb\rb
    =
    H^0 \lb X_{\Sigma'}, kL_h \otimes \scO \lb-\sum_{j \in J} \sum_{n \in \nabla_j \cap N \setminus \lc 0 \rc} D_n \rb\rb.
\end{align}
The vector space $H^0 \lb X_{\Sigma'}, kL_h \rb$ is spanned by monomial sections that correspond to the lattice points in the polytope $k\Delta_h$, and \eqref{eq:h0kLX} is the subspace of $H^0 \lb X_{\Sigma'}, kL_h \rb$, which is spanned by monomial sections vanishing on all $D_n$ with $n \in \nabla_j \cap N \setminus \lc 0 \rc, j \in J$.

We use the inclusion–exclusion principle which claims as follows:

\begin{theorem}{\rm(cf.~e.g.~\cite[Theorem 6.1]{MR1797778})}\label{th:i-e}
For finite sets $A_1, \cdots, A_r$, one has
\begin{align}
    \# \lb \bigcup_{i=1}^r A_i \rb
    =\sum_{\emptyset \neq J \subset \lc 1, \cdots, r\rc} (-1)^{\# J-1} 
    \cdot \# \lb \bigcap_{j \in J} A_j \rb.
\end{align}
\end{theorem}

Let  $A_j$ $(j \in \lc 1, \cdots, r\rc)$ be the set of monomial sections of $kL_h$ vanishing on all $D_n$ with $n \in \nabla_j \cap N \setminus \lc 0 \rc$, and apply \pref{th:i-e}.
Then we see that the latter part of \eqref{eq:h0kL} is the number of monomial sections of $kL_h$ vanishing on all $D_n$ with $n \in \nabla_j \cap N \setminus \lc 0 \rc$, for some $j \in \lc1, \cdots, r \rc$.
Thus we obtain
\begin{equation}
h^0 \lb \scX_t, kL_h \rb= \#\left\{
m \in k\Delta_h \cap M\;\left|\;
  \begin{gathered}
    \forall j \in \lc 1, \cdots, r \rc, \exists n \in \nabla_j \cap N \setminus \lc 0\rc \mathrm{\ such\  that\ } \\  s^m \mathrm{\ does\ not\ vanish\ on\ } D_n 
  \end{gathered}
\right.
\right\},
\end{equation}
where $s^m$ is the monomial section of $kL_h$ corresponding to $m \in k\Delta_h \cap M$.
The section $s^m$ does not vanish on the toric divisor $D_n$ if and only if $k \cdot h(n)+\la m, n \ra =0$.
Hence, we obtain
\begin{align}
    h^0 \lb \scX_t, kL_h \rb
    =\# \lb A_{k \cdot h} \cap M \rb
    =\# \lb k \cdot A_{h} \cap M \rb,
\end{align}
where $A_{k \cdot h}$ is $A_h$ of \eqref{eq:A} with $h$ multiplied by $k$.
Recall that the leading term of the Ehrhart polynomial of a lattice polytope is the affine volume of the polytope (cf.~e.g.~\cite[Corollary 3.20, Section 5.4]{MR3410115}).
By applying this fact to every facet of $A_h$, one gets
\begin{align}\label{eq:lattice-pt}
\# \lb k \cdot A_{h} \cap M \rb
\sim
k^d \cdot \mu(A_h)
\end{align}
for $k \gg 1$.

On the other hand, by the asymptotic Riemann--Roch formula, we have
    \begin{align}
        h^0 \lb \scX_t, k L_h \rb \sim \frac{ (L_h^d) }{d!} k^d
    \end{align}
    for $k \gg 1$.
    By comparing this and \eqref{eq:lattice-pt}, we obtain \eqref{eq:vol-L}.
\end{proof}

Suppose (i) of \pref{th:main} and 
take the semi-positive toric metric $\Phi$ corresponding to the convex function $\phi \in \scP\lb \Delta_h \rb$ of (i).
By \pref{lm:potential-retraction}(1), we can apply \cite[Theorem 1.1]{MR4204708}. 
We can see that for any facet $\tau$ of $\Sk(\scrX)$, one has
\begin{align}
    \MA \lb \Phi \rb |_{\rho_\scrX^{-1} \lb \rint \lb \tau \rb \rb}=d! \MA_\bR \lb \left. \psi \right|_{\rint (\tau)} \rb.
\end{align}
(We regard the right hand side as a measure on $\rho_\scrX^{-1} \lb \rint \lb \tau \rb \rb$ by taking pushforward by the inclusion $\rint \lb \tau \rb \hookrightarrow \rho_\scrX^{-1} \lb \rint \lb \tau \rb \rb$.)
By using \pref{lm:potential-retraction}(2), \eqref{eq:real-MA} and \pref{lm:vol}, we obtain
\begin{align}
    \MA \lb \Phi \rb|_{\rho_\scrX^{-1} \lb \rint \lb \tau \rb \rb}
    &=d! \MA_\bR \lb \left. \phi \right|_{\rint (\tau)} \rb\\
    &=d! \mu \lb A_h \rb \left. \tilde{\nu} \right|_{\rint (\tau)}\\ \label{eq:facet-measure}
    &=(L_h^d) \left. \tilde{\nu} \right|_{\rint (\tau)},
\end{align}
where $\tilde{\nu}$ is the normalized Lebesgue measure on $B_{\check{h}}$ of \eqref{eq:normalized}.
Here, the total mass of the sum of \eqref{eq:facet-measure} over all the facets $\tau$ of $\Sk (\scrX)$ is $(L_h^d) $, and the total mass of $\MA \lb \Phi \rb$ is also $(L_h^d) $ by \cite[Section 2]{MR2244803}.
Since $ \MA \lb \Phi \rb$ is a positive Radon measure on $X^{\an}$, we can see that the interiors of facets of $\Sk (\scrX)$ contain the full measure of $\MA \lb \Phi \rb$.
Thus we obtain
\begin{align}
    \MA \lb \Phi \rb
    =(L_h^d)  \tilde{\nu}
\end{align}
on $X^{\an}$.
Hence, $\Phi$ is the solution to \eqref{eq:NA-MA}, which implies (ii) of \pref{th:main}. 
Conversely, when we suppose (ii) of \pref{th:main}, the convex function $\phi \in \scP\lb \Delta_h \rb$ corresponding to the semi-positive toric metric $\Phi$ of (ii) of \pref{th:main} satisfies \eqref{eq:real-MA} since 
\begin{align}
    \MA_\bR \lb \left. \phi \right|_{\rint (\tau)} \rb
    = \frac{1}{d!}\MA \lb \Phi \rb|_{\rho_\scrX^{-1} \lb \rint \lb \tau \rb \rb}
    =\frac{(L_h^d)}{d!} \left. \tilde{\nu} \right|_{\rint (\tau)}
    = \mu \lb A_h \rb \left. \tilde{\nu} \right|_{\rint (\tau)}
\end{align}
follows from the same argument as above.
We conclude \pref{th:main}(1).

\section*{Acknowledgement}

The first author is supported by NSTC of Taiwan, with grant number 112-2123-M-002-005.
The second author was supported by RIKEN iTHEMS Program.

\bibliographystyle{amsalpha}
\bibliography{bibs}

\end{document}